\documentclass[10pt]{elsarticle}

\usepackage{amsthm,amsmath,amsfonts,amssymb,amscd,mathrsfs}
\usepackage{txfonts}
\usepackage{supertabular,soul}

\usepackage{tikz,geometry}
\usepackage{multirow}
\usepackage{bbm} %\mathbb numbers and other symbols
\usepackage{hyperref}
\usepackage{overpic}
\usepackage{enumitem}

\newcommand{\comment}[1]{}

\newtheorem{theorem}{Theorem}

\newtheorem{definition}{Definition}
\newtheorem{thm}{Theorem}[section]

\newtheorem{lemma}{Lemma}

\newtheorem{example}[thm]{Example}

\theoremstyle{remark}

\providecommand*{\propertyautorefname}{Property}

\let\oldmarginpar\marginpar
\renewcommand\marginpar[1]{\oldmarginpar[\raggedleft\footnotesize #1]%
{\raggedright\footnotesize #1}}

\begin{document}
\begin{frontmatter}

%\date{\today}

\title{Stability of Stochastically Switched and Stochastically Time-Delayed Systems}
\author[camille]{Camille Carter}
\author[jacob]{Jacob Murri}
\author[david]{David Reber}
\author[ben]{Benjamin Webb}
\address[camille]{Department of Mathematics, Brigham Young University, Provo, UT 84602, USA}
\address[jacob]{Department of Mathematics, Brigham Young University, Provo, UT 84602, USA}
\address[david]{Independent Researcher, davidpreber@mathematics.byu.edu}
\address[ben]{Department of Mathematics, Brigham Young University, Provo, UT 84602, USA, bwebb@mathematics.byu.edu}

\newcommand{\R}{\mathbb{R}}

%\subjclass[2010]{Primary: 14N35, 53D45, Secondary: 32S05, 37K10, 37K20, 35Q53}
% INTRO start

\begin{abstract}
In this paper we introduce the notion of a \textit{patient first-mean stable} system. Such systems are switched systems that are \textit{first-mean stable} meaning that they converge to a globally attracting fixed point on average. They are also patient so that they do not lose their first-mean stability when time-delays are introduced into the system. As time-delays are, in general, a source of instability and poor performance patient first-mean stability is a much stronger condition than first-mean stability. This notion of patient stability allows one to design systems that cannot be destabilized via time-delays. It also significantly reduces the difficulty of modeling such systems since in patient systems time-delays can, to a large extent, be safely ignored. The paper's main focus is on giving a sufficient criteria under which a system is patient first-mean stable and we give a number of examples that demonstrate the simplicity of this criteria.

%Our result leverages the fact that in many real-world systems, identifying the average behavior is more relevant than bounding the worst-case behavior.
%Patient first-mean stability is a much stronger condition than first-mean stability, as it allows practitioners to significantly reduce the difficulty of modeling time delays in stochastic switched systems, whether the system is linear or nonlinear, positive or mixed-sign. Our examples in this paper demonstrate the simplicity of this criteria.
\end{abstract}


\end{frontmatter}

\section{Introduction}

Systems whose governing rules change over time are widely used to model and control real-world systems. For example, daily/annual cycles can be incorporated to refine population models \cite{tv:PeriodicPopulation}.
However, switching can be destabilizing, and traditional stability analysis can become intractable in a switching regime (see \cite{people5} for a detailed overview of the stability theory of switched systems).
In addition, real-world systems are often stochastically time-delayed due to traffic \cite{tv:StochasticTrains}, finite transmission speeds and spatial separation \cite{tv:IntroPaper}.
These delays can have a destabilizing effect on the system \cite{tv:destabilizing1,tv:destabilizing2}.
In addition, many social, biological, information, and technological systems have network-like structure \cite{tv:networks}, where the state space is in fact a product space.
In this paper we propose new tools to analyze discrete-time, delayed, switched systems on a product space.

When considering the stability of a stochastically switched system, one question which arises is whether to focus on bounding all possible trajectories, or only those most frequently encountered. If we want to consider all trajectories then we can, for example, consider the systems the joint spectral radius \cite{p14}, which can be used to determine the stability of linear switched systems \cite{p15}, by bounding the norms of all trajectories (although it is in general very difficult to compute or approximate). On the other hand, $p^{th}$ mean stability \cite{p7} is less demanding, requiring only that, for linear systems, the $p^{th}$ power of the norm of the trajectory converge to $\textbf{0}$ \textit{on average}.

In \cite{pradius} the $p$-radius is shown to characterize $p^{th}$ mean-stability in the case that switching is stochastically independent at each time step. The benefit is that using this criteria stability is relatively easy to compute. In particular, the stability of positive switched systems \cite{p10,p11,p12} is easily analyzed using the $p$-radius.

For delayed systems, the most common approaches for characterizing stability are Lyapunov methods, Linear Matrix Inequalities, and Semi-Definite Programming (see, for example, \cite{tv:referee1,tv:dependent1}. In \cite{tv:BunWebb0,tv:BunWebb2} the notion of intrinsic stability was developed as an alternative to these approaches for studying delayed systems. The approach used there consists of bounding the trajectories of a possibly nonlinear system by a positive linear system, which we refer to here as the associated \emph{Lipschitz} system. If this corresponding Lipschitz system is stable, then \cite{tv:BunWebb0} demonstrated that any delayed version of the nonlinear system will also be stable (including the original undelayed system), so long as the delays involved are constant in time. Thus, one can show intrinsic stability is relatively simple to verify using, for instance, the power method \cite{tv:powermethod}.

The approach of considering the intrinsic stability of a system has recently been extended to switched systems. \cite{timevaryingdelays} demonstrated that intrinsically stable unswitched systems are invariant to switching delays.
\cite{timevaryingdelays} also considered switched systems that are delayed, but since the primary proof technique relied on the joint spectral radius, additional constraints were required which increase the computational complexity of this approach.

In this paper we demonstrate that these issues can be avoided when we use the $1$-radius, which is the $p$-radius for $p=1$. To this end we define a general notion of \emph{patient first-mean stability} as the case when a stochastic switched (and possibly nonlinear) system retains first-mean stability even when time-delayed.
We show that the computational efficiency of the $p$-radius can be extended to nonlinear switched systems (Theorem \ref{thm:2}), and that in the case of the $1$-radius this condition is strong enough to imply the patient first-mean stability of the nonlinear switched system (Theorem \ref{thm:3}). Thus, the difficulty of modeling time delays in stochastic switched systems can be significantly reduced if the systems is known to be patient first mean-stable. That is, if time-delays exist in the systems they need not be modeled in that delays will not qualitatively effect the dynamics of the system and this is true whether the system is linear or nonlinear, positive or mixed-sign.

This paper is organized as follows. Section \ref{sec:2} formally introduces switched systems. Section \ref{sec:3} formalizes time-delayed switched systems, and illustrates some of the dynamics which consequently arise. Section \ref{sec:stability} introduces the tools used to check patient first-mean stability, and presents our main results (Theorems \ref{thm:2} and Theorem \ref{thm:3}).

\section{Switched Systems}\label{sec:2}

\begin{definition}\label{def:1}
{\textbf{\emph{(Dynamical System on a Product Space)}}}
Let $(X_i, \| \cdot \|_{X_i})$ be a Banach space for $1 \leq i \leq n$, and let $d_i$ be the metric induced by $\| \cdot \|_{X_i}$. Let $(X, \| \cdot \|)$ be the Banach space formed by giving the product space $X = \prod_{i=1}^n X_i$ the norm
\begin{align*}
    \|\mathbf{x} \| = \max_{1 \leq i \leq n} \| x_i \| \text{ where } \mathbf{x}=[x_1,x_2,\dots,x_n]^T \in X \text{ and } x_i \in X_i,
\end{align*}
and let $d$ be the metric induced by $\|\cdot\|$ on $X$. Let $F: X \to X$ be a Lipschitz continuous map with $i^{th}$ component $F_i:X\rightarrow X_i$. We refer to the dynamical system $(F,X)$ generated by iterating the function $F$ on $X$ as a \textit{dynamical system on a product space}. For the initial condition $\mathbf{x}^0 \in X$ we define the $k^{th}$ iterate of $\mathbf{x}^0$ to be $\mathbf{x}^k = F^k(\mathbf{x}^0)$, with orbit $\{F^k(\mathbf{x}^0)\}^\infty_{k=0} = \{\mathbf{x}^0, \mathbf{x}^1, \ldots\}$ in which $\mathbf{x}^k$ is the state of the system at time $k \geq 0$.
\end{definition}

The dynamical systems we consider in this paper are all of the form given in Definition \ref{def:1}. Specifically, each is a system consisting of a Lipschitz continuous map $F$ on a product space $X$. In this section we also define the notion of a \textit{switched system} in which, at each time step, we switch between systems of this type. In the following section we investigate how time delays affect a system's dynamics as a special case of these switched systems.

\begin{definition}\label{def:2}
{\textbf{\emph{(Stochastic Switched System)}}}
Let $M$ be a set of Lipschitz continuous mappings on the product space $X$ such that for every $F \in M$, $(F, X)$ is a dynamical system. For a probability space $(M,\mathscr{A}, \mu)$ we call $(M, X, \mu)$ a \emph{stochastic switched system} on $X$. Given some sequence $\{F^{(k)}\}^\infty_{k=1}$, where each $F^{(k)}\in M$ is drawn independently with respect to $\mu$, we say that $(\{F^{(k)}\}^\infty_{k=1}, X)$ is an instance of $(M,X,\mu)$ with orbit determined at time $k$ by $\mathbf{x}^k=\mathscr{F}^k(
\mathbf{x}^0) = F^{(k)} \circ \ldots F^{(2)} \circ F^{(1)}(\mathbf{x}^0)$ for the initial condition $\mathbf{x}^0 \in X$.
\end{definition}

\begin{example}\label{ex:1} Consider the dynamical systems $(F,\mathbb{R}^2)$ and $(G,\mathbb{R}^2)$ given by
\[
F\left(
\begin{bmatrix}
x_1 \\
x_2
\end{bmatrix}
\right) = \begin{bmatrix} \epsilon & 1 \\ 0 & \epsilon \end{bmatrix}
          \begin{bmatrix} \tanh(x_1) \\ \tanh(x_2) \end{bmatrix}
          \hspace{0.5cm}\text{and}\hspace{0.5cm}
G\left(
\begin{bmatrix}
x_1 \\
x_2
\end{bmatrix}
\right) = \begin{bmatrix} \epsilon & 0 \\ 1 & \epsilon \end{bmatrix}
          \begin{bmatrix} \tanh(x_1) \\ \tanh(x_2) \end{bmatrix}
\]
where $\epsilon\in\mathbb{R}$. Note that both $F$ and $G$ are Lipschitz continuous functions on $\mathbb{R}^2$. For $M=\{F,G\}$ and $\epsilon=1/2$ consider the three different switching distributions defined by $\mu_1(F)=1$, $\mu_1(G)=0$; $\mu_2(F)=9/10$, $mu_2(G)=1/10$; and $\mu_3(F)=1/2$, $\mu_3(G)=1/2$. This gives us three different switched systems in which we always select $F$, mostly select $F$, and evenly choose between $F$ and $G$, respectively.

In the  stochastic switched system $(M,\mathbb{R}^2,\mu_1)$ the function $F$ is selected at every time step so there is, in fact, no switching. Therefore there is only one instance of this system given by $(\{F,F,F,\dots\},\mathbb{R}^2)$. It is possible to show that for any initial condition, the state of this system is asymptotic to the origin (see Figure \ref{fig:1}, (left)). Likewise, the switched system where the function $G$ is selected at every time step shares this property.

In the switched system $(M,\mathbb{R}^2,\mu_2)$ the function $F$ is chosen with probability $9/10$ and $G$ is chosen with probability $1/10$ at each time step. This system also has the property that its state is asymptotic to the origin irrespective of its initial condition (see Example \ref{ex:3}), although the convergence to the origin is typically slower than in $(M,\mathbb{R}^2,\mu_1)$ (see Figure \ref{fig:1}, center).

In $(M,\mathbb{R}^2,\mu_3)$ the situation is quite different. Here at each time step we randomly select either the function $F$ or the function $G$ with equal probability. The result is that the system's state no longer converges to any point (see Figure \ref{fig:1} (right)). That is, if we switch frequently enough between $(F,\mathbb{R}^2)$ and $(G,\mathbb{R}^2)$ the state of the system no longer tends to a single point and the system's dynamics are no longer stable.
\end{example}

\begin{figure}[h!]
\centering
    \begin{overpic}[scale=0.33]{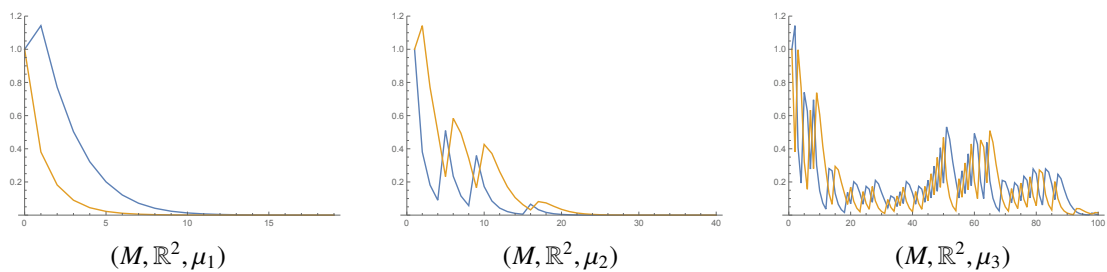}
    \put(10,-2){$(M,\mathbb{R}^2,\mu_1)$}
    \put(45,-2){$(M,\mathbb{R}^2,\mu_2)$}
    \put(80,-2){$(M,\mathbb{R}^2,\mu_3)$}
    \end{overpic}
\vspace{0.25cm}
\caption{The dynamics of the switched system $(M,\mathbb{R}^2,\mu_i)$ from Example \ref{ex:1} is shown left, center, and right for $i=1$, $i=2$, and $i=3$, respectively, for the case in which $\epsilon=1/2$.}\label{fig:1}
\end{figure}

Here we say that a dynamical system is \textit{stable} if it has a globally attracting fixed point. Specifically, $(F,X)$ is \textit{stable} if there is a point $\mathbf{x}^*\in X$ such that
\[
\lim_{k\rightarrow\infty}F^k(\mathbf{x})=\mathbf{x}^* \quad \text{for any} \quad \mathbf{x}^0\in X.
\]
Such a point $\mathbf{x}^*$ is referred to as a \textit{globally attracting fixed point}.

In a switched system the notion of convergence to such a point is more complicated since for some instances the state of the system may evolve toward some $\mathbf{x}^*\in X$ while for others it may not. To illustrate, in Example \ref{ex:1} $(\{F,F,F,\dots\},\mathbb{R}^2)$ is a possible instance of the system $(M,\mathbb{R}^2,\mu_3)$. This particular instance has the globally attracting fixed point $\mathbf{x}^*=(0,0)$ but not every instance of this switched system is asymptotic to this point (see Figure \ref{fig:1} (right)). Later, in Section \ref{sec:stability}, we define the notion of $p^{th}$-mean stability (see Definition \ref{def:pthmeanstable}) which extends the notion of stability to switched systems.

\section{Time-Delayed Systems}\label{sec:3}

For any dynamical system $(F,X)$ one can consider time-delayed versions of the system. In many real-world systems, time delays are an inherent part of the system's dynamics. For example, one can consider the dynamical system $(F,X)$ given in Definition \ref{def:1} to be a model for the dynamics of real-world networks. In this setting the network is comprised of $n$ elements where the dynamics of the $i^{th}$ element is given by the function $F_i:X\rightarrow X_i$. Although the function $F_i=F_i(x_1,\dots,x_n)$ formally depends on each $x_j$ for $j=1,\dots,n$ it may, in fact, only depend on a subset of these variables. The network elements corresponding to this subset are the elements that directly \textit{interact}, or influence the dynamics of the $i^{th}$ network element.

In many real-world networks these interactions do not happen instantaneously but are time-delayed. For example, network elements are often separated by distance, or there may be other processes that must occur before a network element can influence another.

We define a time-delayed version of a dynamical system $(F,X)$ as follows.

\begin{definition}{\textbf{\emph{(Time-Delayed Dynamical System)}}}
\label{def:3}
Let $(F,X)$ be a dynamical system given in Definition \ref{def:1} on the product space $X=\prod_{i=1}^n X_i$. Let $D=[d_{ij}]\in\mathbb{N}^{n\times n}_L$ be a delay matrix with $\max_{i,j}d_{ij}\le L$, a bound on the delay length. Define $X_L$, the \emph{extension of $X$ to delay-space}, to be
\[
X_L = \prod_{\ell=0}^L \prod_{i=1}^n X_{i,\ell }
\quad \text{where} \quad X_{i,\ell}=X_i \quad \text{for} \quad 1 \le i \le n, \, 0 \le \ell \le L.
\]
Componentwise, define $F_D:X_L \to X_L$ by
\begin{equation}\label{identity}
% (F_D)_{i,\ell+1}:X_{i,\ell} \to X_{i,\ell+1} \quad \text{ given by the identity map } \quad
(F_D)_{i,\ell+1}(x_{i,\ell}) = x_{i,\ell}\quad \text{ for } \quad 0 \le \ell \le L-1, \, 1 \leq i \leq n,
\end{equation}
that is the identity map, and
\begin{equation}\label{same-but-later}
%(F_D)_{i,0}:\bigoplus_{j=1}^n X_{j,d_{ij}} \to X_{i,0} \quad \text{ given by } \quad
(F_D)_{i,0}(x_{1,d_{i1}},x_{2,d_{i2}},\hdots,x_{n,d_{in}}) = F_i(x_{1,d_{i1}},x_{2,d_{i2}},\hdots,x_{n,d_{in}})
\end{equation}
where $F_i:X\to X_i$ is the $i^{th}$ component function of $F$ for $i=1,2,\hdots,n$.
Then $(F_D,X_L)$ is the \emph{delayed version} of $(F,X)$ corresponding to the fixed delay matrix $D$ with \emph{delay bound} $L$.
\end{definition}

This definition of a time-delayed dynamical system uses the product structure of $X$ to create the new larger delay space $X_L$. For every dimension of original space $X$ we add $L$ extra dimensions to our product space to allow for the ``storage" of the time-delayed interactions. If the space $X$ has dimension $n$, the space $X_L$ has dimension $n(L+1).$ Viewing the dynamical system as a network sheds some light on this construction. When an interaction between element $i$ and element $j$ is delayed by $\tau$ time-steps (i.e. $d_{ij}=\tau$), the state of the element $i$ does not immediately affect the state of element $j$. Rather the ``effect'' $i$ has on $j$ is passed through $\tau$ ``filler'' elements before reaching element $j$. In other words, the state of the $i$th element $\ell$ time-steps in the past is stored as $x_{i,\ell}$. The information about which interactions are delayed and by how much is stored in the delay matrix $D = [d_{ij}] \in \mathbb{N}_L^{n \times n}$.

\begin{figure}
\centering
    \begin{overpic}[scale=0.3]{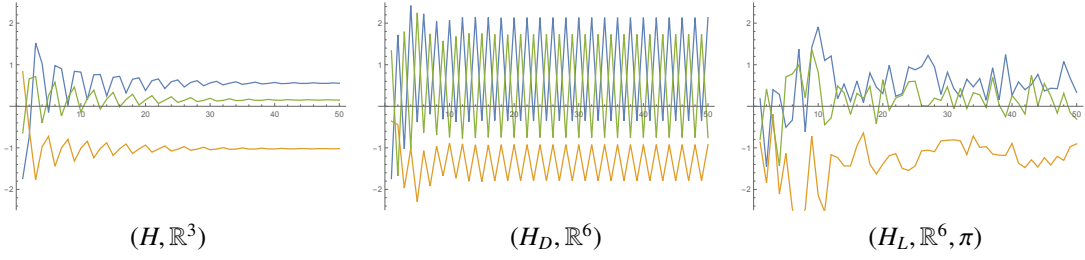}
    \put(12,-2){$(H,\mathbb{R}^3)$}
    \put(47,-2){$(H_D,\mathbb{R}^6)$}
    \put(80,-2){$(H_{L},\mathbb{R}^6,\pi)$}
    \end{overpic}
\vspace{0.25cm}
\caption{The stable dynamics of the system $(H,\mathbb{R}^3)$ from Example \ref{ex:2} is shown left. Adding the constant time delays $D$ given by Equation \ref{eq:3} results in the system $(H_D,\mathbb{R}^6)$ with periodic dynamics, shown center. Adding stochastically time-varying time-delays of size no larger than $L=1$ similarly destabilizes the system's dynamics. This is shown right for the system $(H_{L},\mathbb{R}^6,\pi)$ (see Example \ref{ex:3}) for details.}  \label{fig:2}
\end{figure}

\begin{example}\label{ex:2} \textbf{(Constant Time-Delays)} Consider the dynamical system $(H,\mathbb{R}^3)$ given by
\[
H\left(
\begin{bmatrix}
x_1 \\
x_2 \\
x_3
\end{bmatrix}
\right) = \begin{bmatrix}
\frac{1}{4}x_1 - \tanh(x_1) + \tanh(x_2) - \tanh(x_3)\\
\frac{1}{4}x_2 - \tanh(x_1) + \tanh(x_2) - \tanh(x_3) - 1/2\\
\frac{1}{4}x_3 + \tanh(x_1) - \tanh(x_3)
\end{bmatrix}.
\]
For the delay matrix
\begin{equation}\label{eq:3}
D = \begin{bmatrix}
0 & 1 & 1 \\
0 & 0 & 1 \\
0 & 0 & 0
\end{bmatrix},
\end{equation}
which has delay bound $L = 1$, the delayed version is given by the function $H_D:X_L\rightarrow X_L$ where $X_L=\mathbb{R}^6$, which has the form
\[
H_D\left(
\begin{bmatrix}
x_{1, 0} \\
x_{2, 0} \\
x_{3, 0} \\
x_{1, 1} \\
x_{2, 1} \\
x_{3, 1}
\end{bmatrix}
\right)
=
\begin{bmatrix}
\frac{1}{4}x_{1, 0} - \tanh(x_{1, 0}) + \tanh(x_{2, 1}) - \tanh(x_{3, 1})\\
\frac{1}{4}x_{2, 0} - \tanh(x_{1, 0}) + \tanh(x_{2, 0}) - \tanh(x_{3, 1}) - 1/2\\
\frac{1}{4}x_{3, 0} + \tanh(x_{1, 0}) - \tanh(x_{3, 0}) \\
x_{1, 0} \\
x_{2, 0} \\
x_{3, 0} \\
\end{bmatrix}.
\]
Here the variables $x_{1,1}, x_{2,1}, x_{3,1}$ represent the ``filler" elements which store delay information. %In a delayed interaction, an interaction propagating towards $x_{i, 0}$ will first be stored in $x_{i, 1}$ before being passed to $x_{i, 0}$ on the following time step.

Note that the original undelayed system $(H, \mathbb{R}^3)$ is stable with the globally attracting fixed point $\mathbf{x}\approx[-1.05, 0.62, 0.14]^T$ but the delayed version $(H_D, \mathbb{R}^6)$ is not (see Figure \ref{fig:2} left and center, respectively). That is, a system's dynamics can be destabilized by the introduction of time-delays.
\end{example}

%\begin{example}
%Some unstable systems become stable when time delays are added. For example, if $X = \mathbb{R}^2$ and $F$ is the linear transformation represented by the matrix
%\[
%A = \begin{bmatrix} 1.3 & 1.1 \\ -1.6 & -1.7 \end{bmatrix},
%\]
%the dynamical network is stable and converges to the fixed point $(0, 0)$ because $A$ has spectral radius $\rho(A) = 0.9$. But if we delay the $(1, 0)$ entry, meaning we use the fixed delay distribution
%\[
%D = \begin{bmatrix}
%0 & 0 \\
%1 & 0
%\end{bmatrix},
%\]
%(which has delay bound $L = 1$),
%then the delayed version of $A$ is
%\[
%A_d =
%\begin{bmatrix}
%1.3 & 1.1 & 0 & 0 \\
%0 & -1.7 & -1.6 & 0 \\
%1 & 0 & 0 & 0 \\
%0 & 1 & 0 & 0
%\end{bmatrix},
%\]
%and $X_L = \mathbb{R}^4$ is the delay space. This system is no longer stable becuase $\rho(A_d) \approx 1.97$.
%\end{example}

Not only can time delays be introduced into dynamical systems they can also be introduced into switched systems.

\begin{definition}{\textbf{\emph{(Delayed Stochastic Switched System)}}}
Let $(M, X, \mu)$ be a stochastic switched system with corresponding probability space $(M, \mathscr{A}, \mu)$. For the delay bound $L$ let $M_L$ be the set of all delayed versions of functions in $M$ with maximum delay length of $L$, i.e.
$$M_L = \{F_D : F \in M \text{ and } \ D \in \mathbb{N}^{n \times n}_L\}.$$
Let $(M_L, \mathscr{A}_D, \pi)$ be a probability space such that for a measurable set $E \subset M$, the set of all \textit{delayed versions} of each $F\in E$, given by
$$E_L = \{F_D : F \in E, D \in \mathbb{N}^{n \times n}_L\},$$
we have $\pi(E_L)=\mu(E)$. If this is the case we call the corresponding switched system $(M_L,X_L,\pi)$ a \textit{delayed version} of $(M,X,\mu)$.
% Let $(M, X, \mu)$ be a stochastic switched system with probability space $(M, \mathscr{A}, \mu)$.
% Let $M_L$, the delay set, be defined as $\mathscr{D} = \{ D \in \mathbb{N}^{n \times n}: 0 \leq d_{ij} \leq L\}$, where $L > 0$ is a bound on the delay length. \\

% Define $g: M \times \mathscr{D} \to C^0(X_L) $ by $g((F,D)) = F_D$, where $F_D$ is the delayed version of $F$ corresponding to the delay distribution $D$ and delay bound $L$ defined above.
% Let $M_L = \{ g(F, D): F \in M, D \in \mathscr{D}\}$.
% Let $(M \times \mathscr{D}, \mathscr{A}', \nu)$ be any probability space such that for any event $E \subset \mathscr{A}$, for the set of all delayed versions of networks in $E$ given by $E \times \mathscr{D} = \{ (F,D) : F \in E, D \in \mathscr{D} \}$, we have $E \times \mathscr{D} \in \mathscr{A'}$ and $\mu(E) = \nu(E \times \mathscr{D})$. \\

% Now let $\mathscr{A}_D = \sigma(\{g(E \times \mathscr{D}) : E \times \mathscr{D} \in \mathscr{A}'\})$, and let $\pi$ be the pushforward measure $g_*\nu$ on $M_L$ defined by $\pi(E_D) = \nu(g^{-1}(E_D))$ for every $E_D \in \mathscr{A}_D$ (where $g^{-1}(E_D) = \{(F, D) \in M \times \mathscr{D}: g(F,D) \in E_D) \}$). $(M_L, \mathscr{A}_D, \pi)$ is a probability space, and we call $(M_L, X_L, \pi)$ a delayed stochastic switched system.
\end{definition}

In the delayed version $(M_L,X_L,\pi)$ of the switched system $(M,X,\mu)$ the idea is that we choose at each time step $k\geq 1$ some $F^{(k)}\in M$ and some delay matrix $D^{(k)}\in\mathbb{N}^{n\times n}_L$. The function $F^{(k)}_{D^{(k)}}\in M_L$ is chosen in this system at time $k$ with probability $\pi(F^{(k)}_{D^{(k)}})$. The one stipulation is the probability distribution $\pi$ has the property
\begin{equation}\label{eq:delayed}
\sum_{D^{(k)}\in\mathbb{N}^{n\times n}_L}\pi(F^{(k)}_{D^{(k)}})=\mu(F^{(k)}).
\end{equation}
That is, delayed versions of a function $F\in M$ are as likely in $(M_L,X_L,\pi)$ as $F$ is in the undelayed system $(M,X,\mu)$. This gives us an instance $(\{F^{(k)}_{D^{(k)}}\}_{k=1}^\infty,X_L)$ of the delayed stochastic switched system $(M_L,X_L,\pi)$.

It is also worth noting that if we have a delayed system $(M_L,X_L,\pi)$ then we can uniquely recover the original undelayed system $(M,X,\mu)$. The reason for this is that the probability distribution $\mu$ that determines the frequency at which we choose a function $F\in M$ is given to us in terms of $\pi$ by Equation \eqref{eq:delayed}.

\begin{figure}
\centering
    \begin{overpic}[scale=0.33]{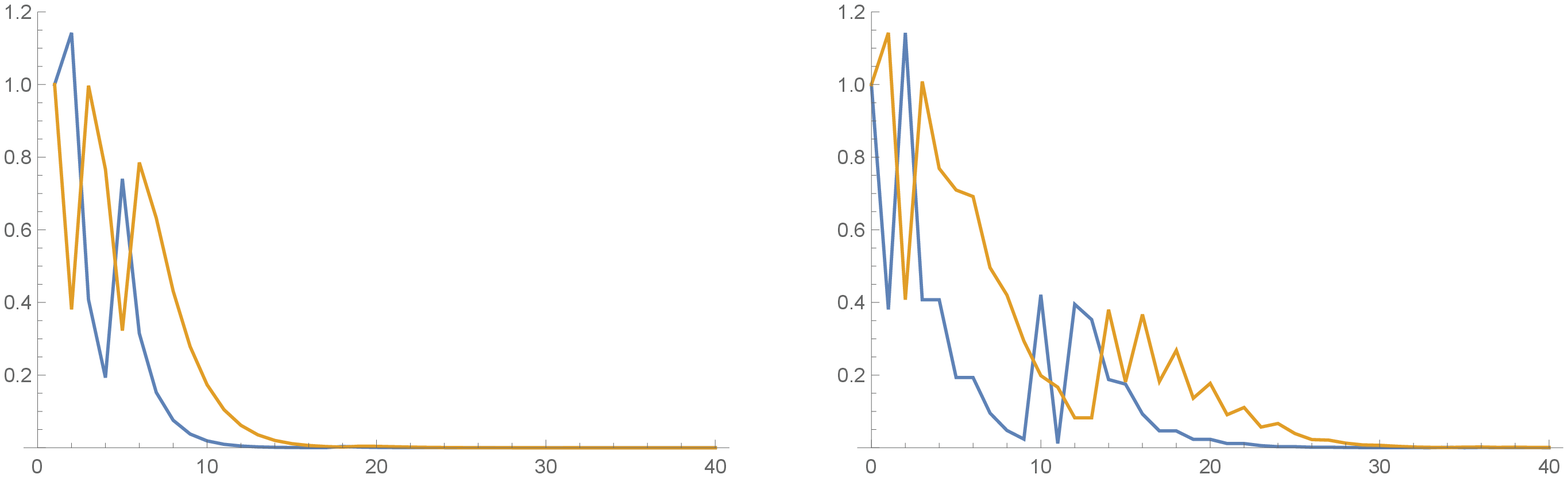}
    \put(20,-2){$(M,\mathbb{R}^2,\mu_2)$}
    \put(70,-2){$(M_L,\mathbb{R}^2,\pi_2)$}
    \end{overpic}
\vspace{0.25cm}
\caption{The undelayed dynamics of the system $(M,\mathbb{R}^2,\mu_2)$ from Example \ref{ex:1} is shown left. The delayed version of this system $(M_L,\mathbb{R}^2,\pi_2)$ for $L=1$ is shown right. Both are stable.}\label{fig:2.1}
\end{figure}

\begin{example} \textbf{(Time-Varying Time-Delays)} Let $M$ be the single mapping $F:X\rightarrow X$, i.e. $M=\{F\}$. Then $(M,X,\mu)$ with $\mu(F)=1$ is a switched system, although there is no system to switch to apart from $(F,X)$. For a delay bound of $L$ the corresponding delayed switched system is $(M_L,X_L,\pi)$ in which an instance $(\{F_{D^{(k)}}\}_{k=1}^\infty,X_L)$ is composed of functions $F_{D^{(k)}}:X_L\rightarrow X_L$ each of which are time-delayed versions of the original function $F$.

Since we switch between different time-delayed versions of the same function $F$ the result is what is referred to as a \textit{dynamical system with time-varying time-delays} where each time-delayed version of $F$ is chosen at each time step with probability determined by $\pi$. We write this system as $(F_D,X_L,\pi)$.

Taking the system $(H,\mathbb{R}^3)$ from Example \ref{ex:2} we can add time-varying time-delays to this system as follows. Let $L=1$. Then there are $2^9$ different delay matrices $D=[d_{ij}]\in\mathbb{N}^{3\times3}_L$, which can be used to create delayed versions $(H_D,\mathbb{R}^6)$ of $(H,\mathbb{R}^3)$. Suppose $\pi$ is the probability distribution that chooses each of these $2^9$ delayed systems with equal probability. The dynamics of the resulting system $(H_L,\mathbb{R}^6,\pi)$ is shown in Figure \ref{fig:2} (right). Note that similar to adding constant time delays adding these time-varying time-delays to the system similarly has a destabilizing effect to the system's dynamics.
\end{example}

Adding time delays can also destabilize switched systems but this is not always the case. For some switched systems it can happen that the addition of time-delays cannot destabilize the system so long as the delay bound $L<\infty$. These ``patient" systems will be introduced later in Section \ref{sec:stability} (see Definition \ref{def:patientstability}).

\begin{example}\label{ex:4}
Consider the switched system $(M, \mathbb{R}^2, \mu_2)$ given in Example \ref{ex:1}. For $L=1$ let $(M_L,\mathbb{R}^4,\pi_2)$ be the delayed version of this system in which the probability distribution $\pi_2$ is given by
$$
\pi(F_D)=\frac{1}{16}\mu_2(F)=\frac{9}{160} \quad \text{and} \quad \pi(G_D)=\frac{1}{16}\mu_2(G)=\frac{1}{160} \quad \text{ for } \text{ any } \quad D\in\mathbb{N}^{2 \times 2}_L.
$$
Note that there are $2^4$ delayed versions of each of $(F,\mathbb{R}^2)$ and $(G,\mathbb{R}^2)$, making for 32 total systems in $M_L$. Here each delayed version of $F$ is equally likely, and each delayed version of $G$ is equally likely, but delayed versions of $F$ are more likely than delayed versions of $G$.

As is shown in Figure \ref{fig:2.1} the delayed switched system $(M_L,\mathbb{R}^4,\pi_2)$ is stable. In fact, any time-delayed version of $(M, \mathbb{R}^2, \mu_2)$ will be stable (see Example \ref{ex:3}). That is, irrespective of the maximal delay length $L<\infty$ and of the way in which we create the probability distribution $\pi$, this switched system cannot be destabilized by time delays.

In contrast, consider the dynamical systems $(\bar{F},\mathbb{R}^2)$ and $(\bar{G},\mathbb{R}^2)$ given by
\[
\bar{F}\left(
\begin{bmatrix}
x_1 \\
x_2
\end{bmatrix}
\right) = \begin{bmatrix} -\epsilon & 1 \\ 0 & -\epsilon \end{bmatrix}
          \begin{bmatrix} \tanh(x_1) \\ \tanh(x_2) \end{bmatrix}
          \hspace{0.5cm}\text{and}\hspace{0.5cm}
\bar{G}\left(
\begin{bmatrix}
x_1 \\
x_2
\end{bmatrix}
\right) = \begin{bmatrix} -\epsilon & 0 \\ -1 & -\epsilon \end{bmatrix}
          \begin{bmatrix} \tanh(x_1) \\ \tanh(x_2) \end{bmatrix}
\]
where $\epsilon = \frac{2}{3}$. Let $\bar{\mu}$ be the distribution given by $\bar{\mu}(F) = \bar{\mu}(G) = \frac{1}{2}$. For $\bar{M} = \{\bar{F}, \bar{G}\}$, the stochastic switched system $(\bar{M}, \mathbb{R}^2, \bar{\mu})$ is stable (see Figure \ref{fig:2.2}, left).

\begin{figure}
\centering
    \begin{overpic}[scale=0.33]{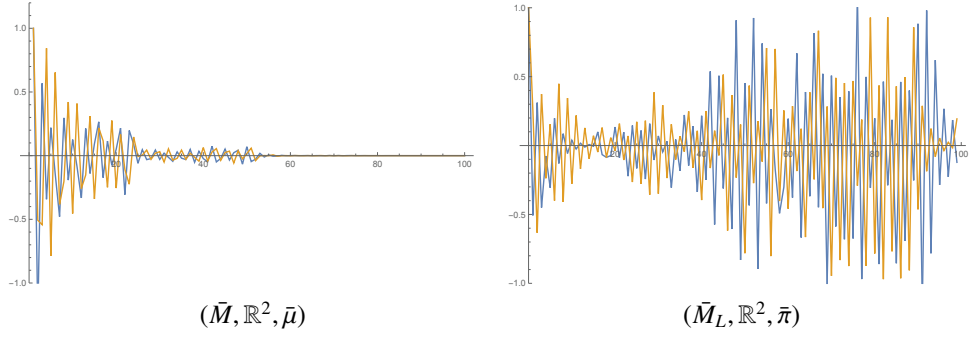}
    \put(20,-2){$(\bar{M},\mathbb{R}^2,\bar{\mu})$}
    \put(70,-2){$(\bar{M}_L,\mathbb{R}^2,\bar{\pi})$}
    \end{overpic}
\vspace{0.25cm}
\caption{The undelayed dynamics of the system $(\bar{M},\mathbb{R}^2,\bar{\mu})$ from Example \ref{ex:4} is shown left. A delayed version of this system $(\bar{M}_L,\mathbb{R}^2,\bar{\pi})$ for $L=1$ is shown right. In this case time-delays destabilize the  switched system.}\label{fig:2.2}

\end{figure}

However, this system can be destabilized by time delays. Consider the delay matrices
\[
D_0 = \begin{bmatrix} 0 & 0 \\ 0 & 0\end{bmatrix} \quad \text{and} \quad D_1 = \begin{bmatrix} 0 & 0 \\ 1 & 0\end{bmatrix},
\]
which correspond to no delays and to a single delay, respectively. Let $\bar{\pi}$ be the probability distribution on $\bar{M}_L$ such that for $H_D \in \bar{M}_L$, we have
\begin{align*}
    \bar{\pi}(H_D) =
    \begin{cases}
    \frac{1}{2} \text{ if } H = \bar{F}, D = D_0  \\
    \frac{1}{20} \text{ if } H = \bar{G}, D = D_0 \\
    \frac{9}{20} \text{ if } H = \bar{G}, D = D_1  \\
     0 \text{ otherwise}.
    \end{cases}
\end{align*}

The stochastically delayed version $(\bar{M}_L, \mathbb{R}^4, \bar{\pi})$ of the system $(\bar{M}, \mathbb{R}^2, \bar{\mu})$
switches between three systems: $\bar{F}, \bar{G}$, and $\bar{G}$ with a single delay. Including this version of $\bar{G}$ with a single delay destabilizes the system (see Figure \ref{fig:2.2}, right). Technically, the system is no longer first mean stable (see Definition \ref{def:pthmeanstable}).
\end{example}

The question we answer in the following sections is why do certain switched systems remain stable when they are delayed while others do not. Specifically, we give a sufficient condition under which a switched system cannot be destablized under the addition of time-delays.

\section{Stability of Stochastically Switched and Delayed Systems}\label{sec:stability}

In this section we examine the $p^{th}$-mean stability of nonlinear stochastic systems. To do this, we create a linear system that bounds the nonlinear system and use tools developed to find stability of linear systems to determine the stability of the nonlinear system.

\begin{definition}{\textbf{\emph{(Lipschitz Matrix)}}}\label{def:5} Let $F: X \to X$ be a Lipschitz continuous mapping on $X$. Then there exist real constants $a_{ij} \geq  0$ such that
\begin{align*}
    \|F_i(x)-F_i(y)\|\leq \sum_{j=1}^n a_{ij} \, \|x_j- y_j\|
\end{align*}
for all $x,y \in X$. We call $A = [a_{ij}] \in \mathbb{R}^{n \times n}$ a \emph{Lipschitz matrix} of the mapping $F$.
\end{definition}

Note that if $F$ is Lipschitz continuous with Lipschitz matrix $A=[a_{ij}]$, and $A \preceq B\in\mathbb{R}^{n\times n}$ entrywise, then $B$ is also a Lipschitz matrix of $F$. That is, $F$ has infinitely many Lipschitz matrices. If $F$ is differentiable, the entrywise smallest Lipschitz matrix $A$ of $F$ is efficiently obtained as
\begin{equation}\label{PartialsStabilityMatrix}
a_{ij} = \sup_{x \in X} \left\lvert\frac{\partial F_i}{\partial x_j}(\mathbf{x})\right\rvert.
\end{equation}
While using entrywise smaller Lipschitz matrices provides tighter bounds on the nonlinear system $F$, the results of this paper apply to any choice of Lipchitz matrix $A$ of $F$.

\begin{definition}\label{def:6}{\textbf{\emph{(Stochastic Lipschitz System)}}}
Let $(M, X, \mu)$ be a stochastic switched system where each $F\in M$ is Lipschitz continuous. We construct a stochastic Lipschitz set $S \subseteq \mathbb{R}^{n \times n}$ as
follows: For each $F \in M$, contribute exactly one Lipschitz matrix $A\in\mathbb{R}^{n\times n}$ of $F$ to the set $S$. Define a probability distribution $\mu'$ on $S$ by setting $\mu'(A) = \mu(F)$ for each $A \in S$. The system $(S,\mathbb{R}^n, \mu')$ is a \textit{linear stochastic switched system}.
\end{definition}

A fact that will be useful to us is that if $F:X\rightarrow X$ is Lipschitz continuous then any delayed version $F_D: X_L\rightarrow X_L$ of this function is also Lipschitz continuous. The following lemma, proved in \cite{timevaryingdelays}, describes how to construct a Lipschitz matrix of a delayed version $F_D$ of $F$ from a Lipschitz matrix of $F$.

\begin{lemma}{\textbf{\emph{(Lipschitz Set for Stochastic Switched Systems)}}}\label{lem:1}
Let $(M,X, \mu)$ be a stochastic switched system with Lipschitz set $S$. For every $A \in S$ and $D\in\mathbb{N}^{n\times n}_L$ define $A_D$ as
\begin{align*}
    A_D =
    \begin{bmatrix}
    A_0 & A_1 & \ldots & A_{L-1} &  A_L\\
    I_n & 0 & \ldots & 0 & 0 \\
    0 & I_n & \ldots & 0 & 0 \\
    \vdots & \vdots & \ddots & \vdots & \vdots \\
    0 & 0 & \ldots & I_n & 0
    \end{bmatrix}\in\mathbb{R}^{n(L+1)\times n(L+1)}
\end{align*}
where $I_n\in\{0,1\}^{n\times n}$ is the identity matrix and each $A_\ell \in \mathbb{R}^{n \times n}$ is defined element-wise by $A_\ell = [a_{ij}\mathbbm{1}_{d_{ij} = \ell}]$, with the indicator function $\mathbbm{1}_{d_{ij} = \ell}$ defined as
\begin{align*}
    \mathbbm{1}_{d_{ij} = \ell} =
    \begin{cases}
    1 \text{ if $d_{ij}=\ell$} \\
    0 \text{ otherwise}
    \end{cases}
    \text{ for } \ 0 \leq \ell \leq L.
\end{align*}
Furthermore, let $S_L = \{ A_D: A \in S, D \in\mathbb{N}^{n\times n}_L\}$. Then each $A_D$ is a Lipschitz matrix of $F_D$ and $S_L$ is a stochastic Lipschitz set of the corresponding delayed system $(M_L, X_L, \pi)$. This forms a delayed linear stochastic switched system $(S_L, \mathbb{R}^{n(L+1)}, \pi')$, where we define $\pi'(A_D) = \pi(F_D)$ for every $F_D \in M_L$ where $A_D$ is the Lipschitz matrix of $F_D$, so that the condition $\sum_{D \in \mathbb{N}^{n \times n}_L} \pi'(A_D) = \pi'(A)$ holds.
\end{lemma}

\begin{figure}
    \centering
    \begin{overpic}[scale=0.33]{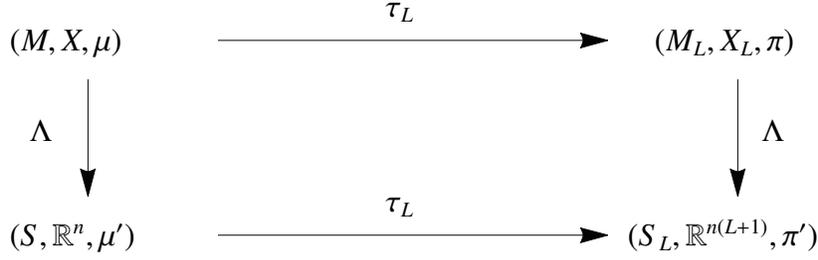}
    \put(-10,29.5){\large{$(M,X,\mu)$}}
    \put(86,29.5){\large{$(M_L,X_L,\pi)$}}
    \put(-10,0.5){\large{$(S,\mathbb{R}^{n},\mu^\prime)$}}
    \put(82,0.5){\large{$(S_L,\mathbb{R}^{n(L+1)},\pi^\prime)$}}

    \put(46,5.5){\large{$\tau_L$}}
    \put(46,34.5){\large{$\tau_L$}}
    \put(-7,16){\large{$\Lambda$}}
    \put(102,16){\large{$\Lambda$}}
    \end{overpic}
\vspace{0.25cm}
\caption{For the delay operator $\tau_L$ and the linearization operator $\Lambda$ defined by Equations ?? the diagram shown above commutes. Therefore, the order in which we linearize and delay the switched system $(M,X,\mu)$ does not matter.}
\label{fig:comm}
\end{figure}

In the following example we show how to construct the matrix $A_D$ of a system delayed by the matrix $D$.

\begin{example}\label{ex:DHmatrix}
Consider the system $(H,\mathbb{R}^3)$ in Example \ref{ex:2}, which can be written in vector form as
\[
H\left(
\begin{bmatrix}
x_1 \\
x_2 \\
x_3
\end{bmatrix}
\right) = \frac{1}{4}\begin{bmatrix} x_1 \\ x_2 \\ x_3\end{bmatrix}+
            \begin{bmatrix} -1 & 1 & -1\\
                            -1 & 1 & -1\\
                             1 & 0 & -1
          \end{bmatrix}
          \begin{bmatrix} \tanh(x_1) \\ \tanh(x_2) \\ \tanh(x_3)\end{bmatrix}-
          \begin{bmatrix} 0 \\ 1/2 \\ 0\end{bmatrix}.
\]
A Lipschitz matrix $A=[a_{ij}]\in\mathbb{R}^{3\times 3}$ for the function $H:\mathbb{R}^3\rightarrow\mathbb{R}^3$ using $\sup_{\mathbf{x}\in\mathbb{R}^3}|\partial H_i/\partial x_j(\mathbf{x})|=a_{ij}$ is
\[
A = \begin{bmatrix}
3/4 & 1 & 1 \\
1 & 5/4 & 1 \\
1 & 0 & 3/4
\end{bmatrix}.
\]
Using the delay matrix $D\in\mathbb{N}^{3\times 3}_1$ from Equation \eqref{eq:3} the delayed version $F_D:\mathbb{R}^6\rightarrow\mathbb{R}^6$ can be written in vector form as
\[
H_D\left(
\begin{bmatrix}
x_{1, 0} \\
x_{2, 0} \\
x_{3, 0} \\
x_{1, 1} \\
x_{2, 1} \\
x_{3, 1}
\end{bmatrix}
\right)
=
\frac{1}{4}
\begin{bmatrix}
1 & 0 & 0 & 0 & 0 & 0 \\
0 & 1 & 0 & 0 & 0 & 0 \\
0 & 0 & 1 & 0 & 0 & 0 \\
0 & 0 & 0 & 0 & 0 & 0 \\
0 & 0 & 0 & 0 & 0 & 0 \\
0 & 0 & 0 & 0 & 0 & 0 \\
\end{bmatrix}
\begin{bmatrix}
x_{1, 0} \\
x_{2, 0} \\
x_{3, 0} \\
x_{1, 1} \\
x_{2, 1} \\
x_{3, 1}
\end{bmatrix}
+
\begin{bmatrix}
-1 & 0 & 0 & 0 & 1 & -1 \\
-1 & 1 & 0 & 0 & 0 & -1 \\
1 & 0 & -1 & 0 & 0 & 0 \\
1 & 0 & 0 & 0 & 0 & 0 \\
0 & 1 & 0 & 0 & 0 & 0 \\
0 & 0 & 1 & 0 & 0 & 0 \\
\end{bmatrix}
\begin{bmatrix}
\tanh(x_{1, 0}) \\
\tanh(x_{2, 0}) \\
\tanh(x_{3, 0}) \\
x_{1, 1} \\
x_{2, 1} \\
x_{3, 1}
\end{bmatrix}
-
\begin{bmatrix}
0 \\
1/2 \\
0 \\
0 \\
0 \\
0
\end{bmatrix}.
\]
From Lemma \ref{lem:1} we can create the Lipschitz matrix $A_D$ of $H_D$ where
\[
A_D = \left[\begin{array}{ccc|ccc}
3/4 & 0 & 0 & 0 & 1 & 1 \\
1 & 5/4 & 0 & 0 & 0 & 1 \\
1 & 0 & 3/4 & 0 & 0 & 0 \\
\hline
1 & 0 & 0 & 0 & 0 & 0 \\
0 & 1 & 0 & 0 & 0 & 0 \\
0 & 0 & 1 & 0 & 0 & 0 \\
\end{array}\right]
\text{ in which }
A_0 = \begin{bmatrix}
3/4 & 0 & 0  \\
1 & 5/4 & 0  \\
1 & 0 & 3/4
\end{bmatrix}
\text{ and }
A_1 = \left[\begin{array}{ccc}
0 & 1 & 1 \\
0 & 0 & 1 \\
0 & 0 & 0
\end{array}\right].
\]
\end{example}

In Section \ref{sec:2} we said that a system $(F,X)$ is \textit{stable} if it has a globally attracting fixed point. Stability is more complicated for a switched system since there may be instances for which the system's state tends towards a single point, irrespective of its initial condition, but not for other instances. To define stability for a stochastically switched system we use the notion of $p^{th}$-mean stability.

\begin{definition}
\label{def:pthmeanstable}
{\textbf{\emph{($p^{th}$-mean stability)}}} We say that a stochastic switched system $(M, X, \mu)$ is exponentially stable in $p^{th}$ mean to a fixed point $\mathbf{\tilde{x}}$ if there exist real constants $C > 0$ and $\beta > 0$ such that for any initial condition $\mathbf{x}^0$ we have
\begin{align*}
    \mathbb{E}\left[||\mathbf{x}^k-\mathbf{\tilde{x}}||^p\right] \leq Ce^{-\beta k}||\mathbf{x}^0-\mathbf{\tilde{x}}||^p.
\end{align*}
\end{definition}

For simplicity, we say that such a system is \textit{$p^{th}$-mean stable}. Though it may be hard to verify that a nonlinear system is $p^{th}$-mean stable  we can use a generalization of the spectral radius called the $p$-radius to determine the stability of an associated linear stochastic switched system.

% add conditions on mu into this section (bounded support)
\begin{definition}
{\textbf{\emph{($p$-radius)}}}
\label{def:8}
For a linear stochastic switched system $(S, \mathbb{R}^n, \mu)$ (where $\mathbf{x}^{k+1} = A_k\mathbf{x}^k$ and each $A_k$ is drawn from $S$ with distribution $\mu$) and constant $p > 0$, the generalized joint spectral radius of $(S, \mathbb{R}^n, \mu)$ with exponent $p$ ($p$-radius for short) is
defined by
\begin{align*}
    \rho_{p, \mu} = \lim_{k \to \infty}(\mathbb{E}[||A_k\cdots A_1||]^p)^{1/pk}.
\end{align*}
\end{definition}

We cite a theorem from \cite{pradius} that says that under certain conditions on $S$ and $\mu$, the $p$-radius can be computed as the spectral radius of a single matrix. Sets of nonnegative matrices, such as the Lipschitz sets we are working with, satisfy these conditions.

Given a probability distribution $\mu$, we define the support of $\mu$ as the closed set supp $\mu$ such that \[\mu((\text{supp }\mu)^c) = 0\] and for any open set $G$ with $G \cap \text{supp }\mu \neq \emptyset$, we have $\mu(G \cap \text{supp } \mu) > 0$.

Next we given some definitions related to Kronecker products. The \emph{Kronecker product} $A \otimes B$ of two matrices $A \in \mathbb{R}^{m \times n}$ and $B \in \mathbb{R}^{q \times r}$ is defined as the $mq \times nr$ matrix
\begin{align*}
    A \otimes B =
    \begin{bmatrix}
    a_{11}B & \cdots & a_{1n}B \\
    \vdots & \ddots & \vdots \\
    a_{m1}B & \cdots & a_{mn}B
    \end{bmatrix}.
\end{align*}
The \emph{Kronecker power} $A^{\otimes p}$ is defined inductively by $A^{\otimes 1} = A$ and $A^{\otimes p} = A^{\otimes(p-1)} \otimes A$ for any integer $p > 1$.

\begin{theorem}{\textbf{\emph{(Computing $p$-radius)}}}
\label{thm:1}
If the support of
$\mu$ is bounded, and the pair $(\mu, p)$ satisfies either
\begin{enumerate}[label=(\arabic*)]
\item $p$ is an integer and $\text{supp }\mu$ leaves a proper cone invariant,
or
\item $p$ is an even integer
\end{enumerate}
then the $p$-radius of $(S, \mathbb{R}^n,\mu)$ is given by
\begin{align*}
    \rho_{p, \mu} = \rho(\mathbb{E}_\mu[S^{\otimes p}])^{1/p},
\end{align*}
where $S^{\otimes p} = \{ A^{\otimes p}: A \in S\}$
and the expectation is taken entrywise according to the distribution $\mu$ on $S$.
\end{theorem}

Now, we relate the stability of a nonlinear system to that of its corresponding Lipschitz system.

\begin{theorem}{\emph{\textbf{(Lipschitz Linearization)}}}
\label{thm:2}
Let $(M,X,\mu)$ be a stochastic switched system. If each $F \in M$ shares a fixed point $\mathbf{\tilde{x}}$, and there exists a corresponding Lipschitz system $(S, \mathbb{R}^n, \mu')$ with $\rho_{\mu,p}(S) < 1$, then $(M, X, \mu)$ is exponentially stable in $p^{th}$ mean to the fixed point $\mathbf{\tilde{x}}$.
\end{theorem}

\begin{proof} Assume the hypothesis. Definition \ref{def:2} implies $\mathbf{\tilde{x}}$ is a fixed point of $(M, X, \mu)$.
Since $\rho_{\mu,p}(S) < 1$, there exist real constants $C > 0$ and $\beta > 0$ such that for any initial condition $\mathbf{y}^0\in\mathbb{R}^n$ we have $\mathbb{E}\left[||\mathbf{y}^k||^p\right] \leq Ce^{-\beta k}||\mathbf{y}^0||^p$.
Let $\mathbf{x}^0\in X$ be given.
Set
\[
\mathbf{y}^0=
\begin{bmatrix}
||\mathbf{x}^0_1-\mathbf{\tilde{x}}_1|| \\
||\mathbf{x}^0_2-\mathbf{\tilde{x}}_2|| \\
\vdots \\
||\mathbf{x}^0_n-\mathbf{\tilde{x}}_n||
\end{bmatrix}.
\]
Then $||\mathbf{x}^0-\mathbf{\tilde{x}}||=||\mathbf{y}^0||$ by Definition \ref{def:1} and $||\mathbf{x}^1-\mathbf{\tilde{x}}||\le||\mathbf{y}^1||$ by Definitions \ref{def:1} and \ref{def:5}.
Thus $\mathbb{E}_\mu\left[||\mathbf{x}^1-\mathbf{\tilde{x}}||^p\right]\le\mathbb{E}_{\mu'}\left[||\mathbf{y}^1||^p\right]$ by Definition \ref{def:6}.
Inductively we have
\begin{align*}
    \mathbb{E}_\mu\left[||\mathbf{x}^k-\mathbf{\tilde{x}}||^p\right] &\le \mathbb{E}_{\mu'}\left[||\mathbf{y}^k||^p\right] \\
    &\le Ce^{-\beta k}||\mathbf{y}^0||^p \\
    &= Ce^{-\beta k}||\mathbf{x}^0-\mathbf{\tilde{x}}||^p
\end{align*}
so $(M, X, \mu)$ is exponentially stable in $p^{th}$ mean to $\mathbf{\tilde{x}}$ as claimed.
\end{proof}

Now that we have established tools for determining first-mean stability of nonlinear switched systems, we focus on determining which of these systems will remain stable even when time delays are introduced. We refer to such systems as ``patiently stable'', as they cannot be destabilized by bounded time-delays of any kind.

\begin{definition}
\label{def:patientstability}
\textbf{(Patient First Mean Stability)}
A stochastic switched system $(M, X, \mu)$ is \emph{patiently first-mean stable} if $(M, X, \mu)$ as well as all of its delayed versions $(M_L, X_L, \pi)$ are first-mean stable.
\end{definition}

Our main result (Theorem \ref{thm:3}) gives a sufficient condition for patient first mean stability. In order to prove our main result (see Theorem \ref{thm:3}) we will need to relate the spectral radius of the matrix $A_D$ in Lemma \ref{lem:1} with the spectral radius of the matrix $B=\sum_{\ell=0}^L A_\ell$. To do this we use a matrix transform referred to as an isoradial matrix reduction, or the Perron compliment \cite{Mey89}. An isoradial reduction is a special type of isospectral reduction that preserves the matrix' spectral radius while reducing its size. (See \cite{Smith19,tv:BunWebb2} for details.)

To define an isoradial matrix reduction we require the following. For a matrix $M\in\mathbb{R}^{n\times n}$ let $N=\{1,\ldots,n\}$. If the sets $R,C\subset N$ are proper subsets of $N$, we denote by $M_{RC}$ the $|R| \times |C|$ \emph{submatrix} of $M$ with rows indexed by $R$ and columns indexed by $C$. We denote the subset of $N$ not contained in $S$ by $\bar{S}$, that is $\bar{S}$ is the \emph{complement} of $S$ in $N$. Using this notation, an isoradial reduction of a square real-valued matrix is defined as follows.

\begin{definition} \textbf{(Isoradial Reductions)}
The \emph{isoradial reduction} of a matrix $M\in\mathbb{R}^{n\times n}$ over a nonempty subset $S\subset N$ is the matrix
\begin{equation}\label{eqn:IRR}
\mathcal{I}_{S}(M)=M_{SS}-M_{S\bar{S}}\left(M_{\bar{S}\bar{S}}-\rho(M) I\right)^{-1}M_{\bar{S}S}\in\mathbb{R}^{|S|\times|S|},
\end{equation}
where $\rho(M) = \max \{|\lambda |: \lambda \textrm{ is an eigenvalue of } M\}$ is the \emph{spectral radius} of $M$.
\end{definition}

The isoradial reduction $\mathcal{I}_S(M)$ does not exist for every square real-valued matrix $M$ and every subset $S \subseteq N$, due to the fact that the inverse taken in Equation \eqref{eqn:IRR} may not exist. If the isoradial reduction $\mathcal{I}_{S}(M)$ exists then
\begin{equation}\label{eq:isorad}
\rho(\mathcal{I}_{S}(M))=\rho(M).
\end{equation}

Using equation \eqref{eq:isorad} we can prove the following lemma.

\begin{lemma}\label{reduce} Assume $A\in\mathbb{R}^{n(L+1) \times n(L+1)}$ is a nonnegative matrix given in Lemma \ref{lem:1} of the form
\begin{align*}
    A =
    \begin{bmatrix}
    A_0 & A_1 & \ldots & A_{L-1} & A_L \\
    I_n & 0 & \ldots & 0 & 0 \\
    0 & I_n & \ldots & 0 & 0 \\
    \vdots & \vdots & \ddots & \vdots & \vdots \\
    0 & 0 & \ldots & I_n & 0
    \end{bmatrix}
\end{align*}
where each $A_{ij}\in\mathbb{R}^{n\times n}$ and $I_n\in\mathbb{R}^{n\times n}$ is the identity matrix.
Then $\rho(A)<1$ if and only if $\rho(\sum_{\ell=0}^L A_\ell)<1$.
\end{lemma}

\begin{proof}
Let $\rho(A)=\alpha$ and suppose $0<\alpha<1$. We choose $S=\{1,2,\dots,n\}$ in which case
\[
A_{SS}=A_0, \quad
A_{S\bar{S}}=
\begin{bmatrix}
    A_1 & \ldots & A_{L-1} & A_L
\end{bmatrix}, \quad
A_{\bar{S}\bar{S}}=
\begin{bmatrix}
     0 & \ldots & 0 & 0 \\
     I_n & \ldots & 0 & 0 \\
     \vdots & \ddots & \vdots & \vdots \\
     0 & \ldots & I_n & 0
\end{bmatrix}, \text{ and }
A_{\bar{S}S}=
\begin{bmatrix}
    I_n\\
    0\\
    \vdots\\
    0
    \end{bmatrix}.
\]
The isoradial reduction of $A$ over the set $S$ is then
\begin{align*}
\mathcal{I}_{S}(A)=&A_{SS}-A_{S\bar{S}}\left(A_{\bar{S}\bar{S}}-\rho(A) I\right)^{-1}M_{\bar{S}S}\\
                  =&A_0-
\begin{bmatrix}
A_1 & \ldots & A_{L-1} & A_L
\end{bmatrix}
\begin{bmatrix}
-\alpha I_n & \ldots & 0 & 0 \\
I_n & -\alpha I_n  & 0 & 0 \\
\vdots & \ddots & \ddots & \vdots \\
0 & \ldots & I_n & -\alpha I_n
\end{bmatrix}^{-1}
\begin{bmatrix}
    I_n\\
    0\\
    \vdots\\
    0
    \end{bmatrix}.
\end{align*}
Since $\alpha>0$ the matrix $A_{\bar{S}\bar{S}}-\alpha I$ is invertible with inverse
\[
\begin{bmatrix}
-\alpha I_n & \ldots & 0 & 0 \\
I_n & -\alpha I_n  & 0 & 0 \\
\vdots & \ddots & \ddots & \vdots \\
0 & \ldots & I_n & -\alpha I_n
\end{bmatrix}^{-1}=
\begin{bmatrix}
-\alpha^{-1} I_n & 0 & 0 & \ldots & 0 \\
-\alpha^{-2} I_n & -\alpha^{-1} I_n  & 0 & \ldots & 0 \\
-\alpha^{-3} I_n & -\alpha^{-2} I_n  & -\alpha^{-1} I_n & \ldots & 0 \\
\vdots & \vdots & \vdots & \ddots & \vdots \\
-\alpha^{-L} & -\alpha^{-L+1} I_n & -\alpha^{-L+2} I_n & \ldots & -\alpha^{-1} I_n
\end{bmatrix}.
\]
Combining this with the previous equation we have
\begin{align*}
\mathcal{I}_{S}(A)=&A_0-
\begin{bmatrix}
A_1 & \ldots & A_{L-1} & A_L
\end{bmatrix}
\begin{bmatrix}
-\alpha^{-1} I_n & 0 & 0 & \ldots & 0 \\
-\alpha^{-2} I_n & -\alpha^{-1} I_n  & 0 & \ldots & 0 \\
-\alpha^{-3} I_n & -\alpha^{-2} I_n  & -\alpha^{-1} I_n & \ldots & 0 \\
\vdots & \vdots & \vdots & \ddots & \vdots \\
-\alpha^{-L} & -\alpha^{-L+1} I_n & -\alpha^{-L+2} I_n & \ldots & -\alpha^{-1} I_n
\end{bmatrix}
\begin{bmatrix}
    I_n\\
    0\\
    0\\
    \vdots\\
    0
    \end{bmatrix}\\
=& A_0 + \sum_{\ell=1}^L\alpha^{-\ell}A_{\ell}=\sum_{\ell=0}^L\alpha^{-\ell}A_\ell.
\end{align*}
By Equation \eqref{eq:isorad} we have $\rho(\sum_{\ell=0}^L\alpha^{-\ell}A_\ell)=\rho(A)=\alpha$. Note that entrywise the matrix $\sum_{\ell=0}^L A_\ell\preceq\sum_{\ell=0}^L\alpha^{-\ell}A_\ell$ since each $A_i$ is nonnegative and $\alpha<1$. By Proposition 3.3 in \cite{tv:BunWebb2} it follows that
\begin{equation}\label{eq:firstin}
\rho\left(\sum_{\ell=0}^L A_\ell\right))\leq\rho\left(\sum_{\ell=0}^L\alpha^{-\ell}A_\ell\right)=\rho(A)<1.
\end{equation}
Therefore, if $0<\rho(A)<1$ then $\rho(\sum_{\ell=0}^L A_\ell)<1$.

If it is the case that $\rho(A)=0$ then the claim is that $\rho(\sum_{\ell=0}^L A_\ell)=0$. To see this suppose $\rho(A)=0$. For $\epsilon>0$ let $A_{\epsilon}$ be the matrix $A$ where we replace $A_0\in\mathbb{R}^{n\times n}$ with the matrix $A_0+\epsilon I\in\mathbb{R}^{n\times n}$. Since $0\preceq A_{\epsilon}\preceq A+\epsilon I$ and $\rho(A+\epsilon I)=\rho(A)+\epsilon=\epsilon$ then Proposition 3.3 in \cite{tv:BunWebb2} implies $\rho(A_\epsilon)\leq\epsilon$. Choosing $\epsilon<1$ then
\[
\rho(\sum_{\ell=0}^L A_\ell)\leq\rho\left((A_0+\epsilon I)+\sum_{\ell=1}^L A_{\ell}\right)\leq\rho(A_\epsilon)\leq\epsilon
\]
where the first inequality follows from Proposition 3.3 and the second from Equation \eqref{eq:firstin} using $A_{\epsilon}$ instead of $A$. Letting $\epsilon\rightarrow0^+$ it follows that $\rho(\sum_{\ell=0}^L A_\ell)=0$. Thus, if $\rho(A)=0$ then $\rho(\sum_{\ell=0}^L A_\ell)=0$. Combining this with the previous result, if $\rho(A)<1$ then $\rho(\sum_{\ell=0}^L A_\ell)<1$.

To prove the converse, note that if $\alpha\geq 1$ then entrywise $\rho(\sum_{\ell=0}^L\alpha^{-\ell}A_\ell)\preceq\sum_{\ell=0}^L A_\ell$ and Proposition 3.3 in \cite{tv:BunWebb2} implies that
\[
\rho\left(\sum_{\ell=0}^L A_\ell\right)\geq\rho\left(\sum_{\ell=0}^L\alpha^{-\ell}A_\ell\right)=\alpha\geq1
\]
so that $\rho(\sum_{\ell=0}^L A_\ell)\geq 1$. Therefore, if $\rho(\sum_{\ell=0}^L A_\ell)<1$ then it is not possible for $\alpha\geq1$ implying that if $\rho(\sum_{\ell=0}^L A_\ell)<1$ then $\rho(A)<1$. This completes the proof.
\end{proof}

For example, the matrix $A_D$ in Example \ref{ex:DHmatrix} has the spectral radius $\rho(A_D)=2.0239$, which is greater than 1. Hence, Lemma \ref{reduce} implies the same is true of $\rho(\sum_{\ell=0}^1 A_\ell)=(6+\sqrt{17})/4$.

\begin{theorem}\label{thm:3}
% {\textbf{\emph{(Stability of Delayed Stochastic Switched Systems)}}} Let $(M, X, \mu)$ be a stochastic switched system where the support of $\mu$ is bounded and each $F \in M$ shares a fixed point $\tilde{x}\in X$. If there is a Lipschitz system $(S, \mathbb{R}^n, \mu')$ of $(M, X, \mu)$ that is exponentially stable in first-mean, then $(M, X, \mu)$ is patiently first-mean stable.

% Conversely, if there is a stochastically delayed version of $(M, X, \mu)$ whose Lipschitz linearization is exponentially stable in first mean, then the undelayed Lipschitz system $(S, \mathbb{R}^n, \mu')$ (and therefore the original system $(M, X, \mu)$) is also exponentially stable in first mean. Furthermore, $(M, X, \mu)$ is patiently first mean stable.

Let $(M, X, \mu)$ be a stochastic switched system where the support of $\mu$ is bounded and each $F \in M$ shares a fixed point $\tilde{x}\in X$. If there is a Lipschitz system $(S, \mathbb{R}^n, \mu')$ of $(M, X, \mu)$ that is exponentially stable in first-mean, or if there is a delayed version $(M_L, X_L, \mu)$ with corresponding Lipschitz system $(S_L, \mathbb{R}^{(L(n+1)}, \pi')$ that is exponentially stable in first-mean, then $(M, X, \mu)$ is patiently first-mean stable.
\end{theorem}

Before we give a proof of Theorem \ref{thm:3}, we recall the systems in Example \ref{ex:1} and now answer the question of why some of these systems can be destabilized by time delays and others remain stable when time delayed.

\begin{example}\label{ex:3}
We return to the systems defined in Example \ref{ex:1} and examine their stability under stochastic delays. For $M = \{F, G\}$, we have Lipschitz set $S = \left\{\begin{bmatrix} \epsilon & 1 \\ 0 & \epsilon \end{bmatrix}, \begin{bmatrix} \epsilon & 0 \\ 1 & \epsilon \end{bmatrix}\right\}$ with $\epsilon=1/2$. For the system $(M,X,\mu_1)$ with no switching we have
\begin{align*}
    E_{\mu_1}[S] = \begin{bmatrix} \epsilon & 1 \\ 0 & \epsilon \end{bmatrix}
\end{align*}
and $\rho(E_{\mu_1}[S]) = 1/2<1$. Thus, $(M, X, \mu_1)$ will be stable under any stochastic switching between time-delayed versions of $F$ with bounded time-delays.

For $(M,X,\mu_2)$, we have
\begin{align*}
    E_{\mu_2}[S] = \begin{bmatrix} \epsilon & 9/10 \\ 1/10 & \epsilon \end{bmatrix}
\end{align*}
for which $\rho(E_{\mu_2}[S]) = 4/5<1$, so any stochastically time-delayed version of $(M, X, \mu_2)$ is first-mean stable, although the convergence of the system is slower than that of $(M, X, \mu_1)$.

Lastly, for $(M,X,\mu_3)$ we have
\begin{align*}
    E_{\mu_3}[S] = \begin{bmatrix} \epsilon & 1/2 \\ 1/2 & \epsilon \end{bmatrix}
\end{align*}
which has spectral radius $\rho(E_{\mu_3}[S]) = 1$, so stochastically time-delayed versions of $(M, X, \mu_3)$ are not guaranteed to be first mean stable.
\end{example}

We now give a proof of Theorem \ref{thm:3}.

\begin{proof}
Let $(M, X, \mu)$ be a stochastic switched system with corresponding Lipschitz system $(S, \mathbb{R}^n, \mu')$ such that the assumptions of Theorem \ref{thm:3} are satisfied. Let $L > 0$ and let $(M_L, \mathbb{R}^{n(L+1)}, \pi)$ be any delayed version of $(M, X, \mu)$.

 We determine the stability of $(M_L, X_L, \mathbb{R}^{n(L+1)})$ using the system's corresponding Lipschitz set $S_L$. We construct $(S_L, \mathbb{R}^{n(L+1)}, \pi')$ from $(S, \mathbb{R},\mu')$ using Lemma \ref{lem:1}. We then use Theorem \ref{thm:1} to determine the stability of $(S_L, \mathbb{R}^{n(L+1)}, \pi')$, which by Theorem \ref{thm:2} tells us the stability of $(M_L, X_L, \pi)$.

 First, we use Lemma \ref{lem:1} to construct the Lipschitz set $(S_L, \mathbb{R}^{n(L+1)}, \pi')$ of $(M_L, X_L, \pi)$ from the Lipschitz set of the undelayed system, $(S, R^n, \mu)$. We have $S_L = \{A_D: A \in S, D \in\mathbb{N}^{n\times n}_L\}$, where each $A_D$ is of the form defined in Lemma \ref{lem:1}.

 Since $S_L$ is a linear system, we determine the 1st mean stability of $S_L$ using the 1-radius (Definition \ref{def:8}). Furthermore, $S_L$ is comprised of nonnegative matrices. The set of nonnegative vectors in $\mathbb{R}^n$, i.e. $\{\mathbf{x} \in \mathbb{R}^n: \mathbf{x} \succeq \mathbf{0}\}$ is a proper cone because it is closed, convex, solid, and pointed. Since $S_L$ is a set of nonnegative Lipschitz matrices, each matrix leaves this set invariant, and therefore satisfies the assumptions of Theorem \ref{thm:1}. We can therefore apply Theorem \ref{thm:1} to compute the 1-radius as $\rho_{1, \pi'}(S_L) = \rho(E_{\pi'}[S_L])$. Thus, we proceed by computing $E_{\pi'}[S_L]$.

Note that each $A_D \in S_L$ has the same block subdiagonal form (as defined in Lemma \ref{lem:1}). The expectation is taken entrywise, so $E[S_L]$ will have the same block subdiagonal form as well, as follows:
\begin{align*}
    E_{\pi'}[S_L] =
    \begin{bmatrix}
    B_0 & B_1 & \ldots & B_{L-1} & B_L\\
    I_n & 0 & \ldots & 0 & 0 \\
    0 & I_n & \ldots & 0 & 0 \\
    \vdots & \vdots & \ddots & \vdots & \vdots \\
    0 & 0 & \ldots & I_n & 0
    \end{bmatrix},
\end{align*}

where each $B_l$ is a $n \times n$ matrix.

$E_{\pi'}[S_L]$ is now in the form required to place an upper bound on $\rho(E_{\pi'}[S_L])$ using Lemma \ref{reduce}, so we compute $\sum_{\ell=0}^L B_{\ell}$ to apply the lemma.

% \begin{align*}
%     \left[\sum_{\ell = 0}^L B_{\ell}\right]_{i,j} =  \sum_{\ell = 0}^L B_{\ell, ij} \\
%     = \sum_{\ell = 0}^L \sum_{A_D \in S_L} a_{ij}1_{d_{ij} = l} \pi(A_D) \\
%     = \sum_{A_D \in S_L} a_{ij} \pi(A_D) \sum_{\ell = 0}^L 1_{d_{ij} = l} \\
%     = \sum_{A_D \in S_L} a_{ij} \pi(A_D) \\
%     = \sum_{A \in S} \sum_{A_D: A = A} a_{ij} \pi(A_D) \\
%     = \sum_{A \in S} a_{ij} \sum_{A_D: A = A} \pi(A_D) \\
%     = \sum_{A \in S} a_{ij} \mu(A)
%     = E_{\mu}[S]
% \end{align*}
% Now, note that

% \begin{align*}
%     = \sum_{A \in S} \sum_{A_D: A = A} a_{ij} \pi(A_D) \\
%     = \sum_{A \in S} a_{ij} \sum_{A_D: A = A} \pi(A_D) \\
%     = \sum_{A \in S} a_{ij} \mu(A)
%     = E_{\mu}[S]
% \end{align*}

Note that given any $A_D \in S_L$, we can uniquely determine which $n \times n$ matrix $A \in S$ that $A_D$ came from.
Thus, we define the mapping $f: S_L \to S$ by $f(A_D) = A$, where $A$ is the unique matrix in $S$ of which $A_D$ is a delayed version. Since $f$ is measurable, we can also use $f$ to define the measure $f_{*}(\pi')$ on $S$ given by
\begin{align*}
   f_{*}(\pi')(E) = \pi'(\left\{A_D \in S_L : f(A_D) \in E)\right\})
   \quad \text{ for } E \in \mathscr{A}.
\end{align*}

What's more, we can show that $f_*\pi' = \mu'$. For any $E \in \mathscr{A}$, we have $$f_*(\pi')(E) = \pi'(\left\{A_D \in S_L : f(A_D) \in E)\right\}) = \pi'(\left\{ A_D: A \in E, D \in \mathbb{N}_L^{n \times n} \right\})
= \pi'(E_L) = \mu'(E).$$.

Then, using a change of variables, we have
\begin{align*}
    \int_{S_L} f(A_D) \: d\pi' = \int_S A \: df_*(\pi') = \int_S A \: d\mu' = E_{\mu'}[S].
\end{align*}

Recall that the $i,j+n\ell$ entry of $A_D \in S_L$ is $f(A_D)_{ij} = a_{ij}$ if $d_{i,j} = \ell$ and 0 otherwise, where $A = [a_{ij}]$ is the unique undelayed version of $A_D$, and $D = [d_{ij}]$ is a delay matrix such that delaying $A$ with $D$ results in $A_D$. We have
\begin{align*}
    \left[\sum_{\ell = 0}^L B_{\ell}\right]_{ij}
    =  \sum_{\ell = 0}^L [B_{\ell}]_{ij} \\
    = \sum_{\ell=0}^L \left[E_{\pi'}[S_L]\right]_{i, j+n\ell} \\
    = \sum_{\ell = 0}^L \int_{S_L} [A_D]_{i,j+n\ell} \: d\pi' \\
    = \sum_{\ell = 0}^L \int_{S_L} \left[f(A_D)\right]_{ij}1_{d_{ij} = l} \: d\pi' \\
    = \int_{S_L} \left[f(A_D)\right]_{ij} \sum_{\ell = 0}^L 1_{d_{ij} = l} \: d\pi' \\
    = \int_{S_L} \left[f(A_D)\right]_{ij}  \: d\pi' \\
    = \int_S a_{ij} \: d\mu' \\
    = \left[E_{\mu'}[S]\right]_{ij}.
\end{align*}

Therefore $\sum_{l=0}^L B_l = E_{\mu'}[S]$. By Lemma \ref{reduce},
\begin{align*}
    E_{\pi'}[S_L] =
    \begin{bmatrix}
    B_0 & B_1 & \ldots & B_{L-1} & B_L\\
    I_n & 0 & \ldots & 0 & 0 \\
    0 & I_n & \ldots & 0 & 0 \\
    \vdots & \vdots & \ddots & \vdots & \vdots \\
    0 & 0 & \ldots & I_n & 0
    \end{bmatrix}.
\end{align*}
has spectral radius less than 1 if and only if $\sum_{l=0}^L B_l = E_{\mu'}[S]$ has spectral radius less than 1. Since we assumed that $(S, \mathbb{R}^n, \mu')$ was exponentially stable in first mean, $\rho_{(1, \mu')}(S) = \rho(E_{\mu'}[S]) < 1.$ We conclude that $\rho(E_{\pi'}[S_L]) = \rho_{1, \pi'}(S_L) < 1$. Thus $(S_L, \mathbb{R}^{n(L+1)}, \pi')$ is exponentially stable in first mean, and by Theorem \ref{thm:2}, so is $(M_L, X_L, \pi)$. That is, if we have a stochastic switched system with a Lipschitz system that is exponentially stable in first mean, any stochastically delayed version of that system will maintain that stability.

Now, suppose that instead of knowing of a first mean stable Lipschitz set for $(M, X, \mu)$, we know that there exists a stochastically delayed version of the system $(M_L, X_L, \pi)$ which has a first mean stable Lipschitz set $(S_L, \mathbb{R}^{n(L+1)}$. Define the set
\begin{align*}
    S = \{f(A_D) : A_D \in S_L\}
\end{align*}

$(S, \mathbb{R}^{n(L+1)}, \mu')$ is a Lipschitz set for $(M, X, \mu)$.
Then, if we proceed with the same process outlined in the first case of the proof, we can determine that $E_{\mu'}[S] < 1$ if and only if $E_{\pi'}[S_L] < 1$, however, we know that $E_{\pi'}[S_L] < 1$, so $(S, \mathbb{R}^{n(L+1)}, \mu')$ is exponentially stable in first mean. We can then apply the first case of the theorem to determine that $(M, X, \mu)$ is patiently stable.

%\begin{align*}
%    B =
%    \begin{bmatrix}
%    0 & 0 & \ldots & 0 & %\mathbb{E}_\mu[S]\\
%    I_n & 0 & \ldots & 0 & 0 \\
%    0 & I_n & \ldots & 0 & 0 \\
%    \vdots & \vdots & \ddots & %\vdots & \vdots \\
%    0 & 0 & \ldots & I_n & 0
%    \end{bmatrix}.\\
% \end{align*}

%But we have $\rho(\mathbb{E}_\mu[S]) = \rho_{1, \mu} (S) < 1$, so $\mathbb{E}_\mu[S]$ is intrinsically stable. The matrix above corresponds to the maximally delayed version of $\mathbb{E}_\mu[S]$, and so by results from the time-varying time delays paper, $B$ is intrinsically stable. So we have $\rho_{1, \pi}(S_L) = \rho(\mathbb{E}_\pi[S_L]) < \rho(B) < 1$.
\end{proof}

\begin{figure}
\begin{center}
    \begin{overpic}[scale=.4]{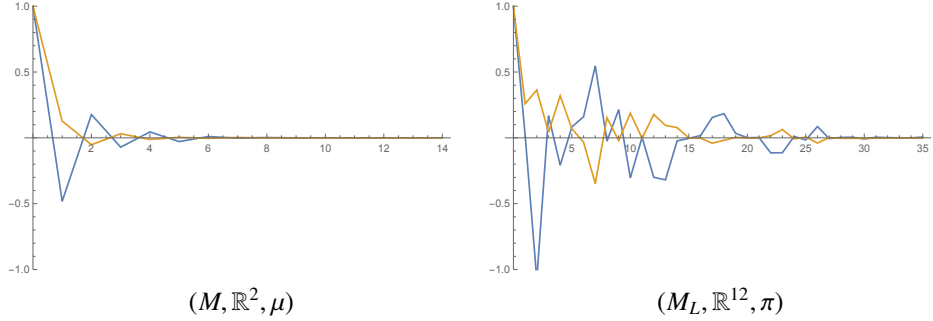}
    \put(20,-2){$(M,\mathbb{R}^{2},\mu)$}
    \put(70,-2){$(M_L,\mathbb{R}^{12},\pi)$}
    \end{overpic}
\vspace{0.25cm}
\caption{The the stable dynamics of the switched system $(M,\mathbb{R}^{2},\mu)$ from Example \ref{ex:4} is shown (left). Since the system is patient first mean stable then any delayed version of the system $(M_L,\mathbb{R}^{2(L+1)},\pi)$ for any $L<\infty$ is also stable. Shown right is the delayed version of this system with $L=5$.}\label{fig:6}
\end{center}
\end{figure}

In all examples given thus far in this paper we have considered only discrete probability spaces. It is worth noting that our results hold for any switched system that satisfies the conditions of Theorem \ref{thm:3}, including those systems that switch over a continuous set of functions. An example of this is the following.

\begin{example}\label{ex:5}
Consider the linear switched system $(M,\mathbb{R}^2,\mu)$ where the set $M\subset\mathbb{R}^{2\times 2}$ is given by
\begin{align*}
M = \begin{bmatrix}
[-0.8, 0] & [0.05, 0.35] \\
[0.05, 0.35] & [0, 0.08] \\
\end{bmatrix}
\end{align*}
where a matrix $A\in M$ has entry $a_{ij}$ drawn independently and uniformly from the interval in the $m_{ij}$ entry of $M$. We obtain a Lipschitz matrix for each $A \in M$ using Equation \ref{PartialsStabilityMatrix}, giving us a Lipschitz set $(S, \mathbb{R^2}, \mu')$ for $M$ where
\begin{align*}
S = \begin{bmatrix}
[0, 0.8] & [0.05, 0.35] \\
[0.05, 0.35] & [0, 0.08] \\
\end{bmatrix}
\end{align*}
and where $\mu'$ is similarly the distribution on $S$ where each element is independently and uniformly distributed on the given interval.

To calculate the one-radius of $S$, we first take the entrywise expectation to get

\begin{align*}
    E[S] = \begin{bmatrix}
        0.4 & 0.2 \\
        0.2 & 0.04 \\
    \end{bmatrix}.
\end{align*}
As $\rho_1(S) = \rho(E[S]) \approx 0.489<1$, the switched system $(M, X, \mu)$ is first-mean stable. Furthermore, by Theorem \ref{thm:3}, $M$ is also patiently first mean stable meaning that any stochastically delayed version of $(M, X, \mu)$ will also be exponentially stable in first mean (see Figure \ref{fig:6}).
\end{example}

We now give a proof that the diagram in Figure 5 commutes.

\begin{proof}
We prove that the diagram shown in Figure 5 commutes for a stochastic switched system $(M, X, \mu)$ in the case where each $F \in M$ is differentiable. In other words, we want to show that given a switched system, it doesn't matter what order in which you delay and 'linearize' the system by creating the Lipschitz set; the result is the same ($\Lambda \circ \tau = \tau \circ \Lambda$). Because we are considering the case where each $F \in M$ is differentiable, we can use Equation $\ref{PartialsStabilityMatrix}$ to define the linearization operator $\Lambda$.

We define $\Lambda$ as
\begin{equation}
    \Lambda(F) = A = [a_{ij}]
    \quad
    \text{ where }
    \quad
    a_{ij} = \sup_{x \in X} \left\lvert\frac{\partial F_i}{\partial x_j}(\mathbf{x})\right\rvert.
\end{equation}

% We define $\Lambda$ as \begin{align*}
%     \Lambda\left((M, X, \mu)\right) = (S, \mathbb{R}^n, \mu') \\
%     \text{ where }
%     S = \left\{ A = [a_{ij}], a_{ij} = \sup_{x \in X} \left\lvert\frac{\partial F_i}{\partial x_j}(\mathbf{x})\right\rvert : F \in M \right\} \\
%     \text{ and }
%     \mu'(E) = \mu\left(\left\{ F \in M: A = [a_{ij}], a_{ij} = \sup_{x \in X} \left\lvert\frac{\partial F_i}{\partial x_j}(\mathbf{x})\right\rvert \in E \right\}\right) \quad \text{ for } E \subset S.
% \end{align*}

The diagram still commutes even when some $F$ is not differentiable, however, the operator $\Lambda$ will be a different bijective function.

We also define the delay operator $\tau_L$ as in Definition \ref{def:3}, i.e., $\tau_L(F) = F_D$ where $F_D$ is the equation defined in Definition \ref{def:3}.

Let $F \in M$. First, we calculate $\Lambda \circ \tau_L (F)$, which corresponds to delaying the system and then computing the Lipschitz linearization.

Recall that for $1 \leq l \leq L, 1 \leq i \leq n$, we have $(F_D)_{i, l}(\mathbf{x}) = x_{i,l-1}.$
So for $1 \leq l \leq L, 1 \leq i \leq n, 0 \leq m \leq L, 1 \leq j \leq n$, we have
\begin{align*}
    [A_D]_{(i,l), (j,m)} = \sup_{x \in X} \left\lvert\frac{\partial F_{D \: i,l}}{\partial x_{j,m}}(\mathbf{x})\right\rvert \\
    = \begin{cases}
    1 \text{ if $i = j, m = l-1$} \\
    0 \text{ otherwise}.
    \end{cases}
\end{align*}

Now, for $\ell = 0$, we have $(F_D)_{i, 0}(\mathbf{x_{\ell}}) =  F_i(x_{1,d_{i1}},x_{2,d_{i2}},\hdots,x_{n,d_{in}})$.

Then

\begin{align*}
    [A_D]_{(i,0), (j,m)} = \sup_{x_{\ell}} \in X_{\ell} \left\lvert\frac{\partial F_{D \: i,0}}{\partial x_{j,m}}(\mathbf{x_{\ell}})\right\rvert \\
    = \sup_{x \in X} \left\lvert\frac{\partial F_i}{\partial x_{j,m}}(x_{1,d_{i1}},x_{2,d_{i2}},\hdots,x_{n,d_{in}})\right\rvert \\
    = \begin{cases}
    a_{ij} \text{ if $m = d_{ij}$} \\
    0 \text{ otherwise}
    \end{cases}.
\end{align*}

Now, we calculate $\tau_L \circ \Lambda (F)$, which corresponds to linearizing the system first and then delaying.

Let $A = \Lambda(F)$, and $A_D$ = $\tau_L(A)$. We have
\begin{align*}
   (A_D)_{i,\ell}(\mathbf{x_{\ell}}) = x_{i,\ell-1}\quad \text{ for } \quad 1 \le \ell \le L, \, 1 \leq i \leq n
\end{align*}

therefore
\begin{align*}
    [A_D]_{(i,l)(j,m)} =
    \begin{cases}
    1 \text{ if $i = j, m = l-1 $} \\
    0 \text{ otherwise}
    \end{cases}.
\end{align*}

And for $l = 0$, we have
\begin{align*}
    (A_D)_{i,0}(\mathbf{x_{\ell}}) = A_i(x_{1,d_{i1}},x_{2,d_{i2}},\hdots,x_{n,d_{in}})
    = a_{i1}x_{1, d_{i1}} + a_{i2}x_{2,d_{i2}} + \hdots +
    a_{in}x_{n, d_{in}}
\end{align*}

so
\begin{align*}
    [A_D]_{(i,0),(j,m)} =
    \begin{cases}
    a_{ij}  \text{ if $m = d_{ij}$} \\
    0 \text{ otherwise }
    \end{cases}.
\end{align*}

Therefore $\tau_L \circ \Lambda (F) = \Lambda \circ \tau_L (F)$.
\end{proof}

\section{Conclusion}\label{sec:conclusion}

In this paper we defined the notion of patient first-mean stability for nonlinear switched systems, and presented a computationally efficient means of checking this form of stability.
It is our hope that practitioners of control and systems theory will apply this theory to novel applications.
We expect that these results will greatly simplify the stability analysis of stochastically switched, time-delayed systems.

A significant research question which remains open is whether the requirement that switching is iid. at each time step can be weakened.
For example, it seems reasonable that similar computationally efficient criteria might exist for the patient $p^{th}$ mean stability of Markov systems, we look forward to further developments in this area.
We also encourage the incorporation of this analysis into open-source software, automating the construction of the Lipschtiz set $S$ and the expectation matrix $E_{\mu'}[S]$ for a given stochastic switched system $(M,X,\mu)$.

%\bibliographystyle{unsrt}

\iffalse

% \bibitem{tv:memory}
% V. Peddinti, D. Povey, and S.Khudanpur
% A time delay neural network architecture for efficient modeling of longtemporal contexts.
% \emph{Sixteenth Annual Conference of the International Speech Communication Association} (2015)

% \bibitem{p1}
% V.Y. Protasov.
% The generalized joint spectral radius: A geometric approach.
% \emph{Izv. Math.} 61 (1997) 995–1030

% \bibitem{p6}
% R.M. Jungers and V.Y. Protasov.
% Weak stability of switching dynamical systems and fast computation of the p-radius of matrices.
% \emph{49th IEEE Conference on Decision and Control} (2010) 7328–7333

% \bibitem{p8}
% D.-X. Zhou.
% The p-norm joint spectral radius for even integers. \emph{Methods Appl. Anal.} \textbf{5} (1998) 39–54

% \bibitem{p9}
% V.Y. Protasov.
% When do several linear operators share an invariant cone?
% \emph{Linear Algebra Appl.} \textbf{433} (2010) 781–789

\end{thebibliography}

\fi

\end{document}